\RequirePackage{fix-cm}
\documentclass[smallcondensed]{svjour3}     
\smartqed  
\usepackage{graphicx}
\usepackage{amssymb}
\usepackage{amsmath}

\newcommand\mb{\mathbb}
\newcommand\mc{\mathcal}
\newcommand \Pe{\mc P}
\newcommand\wt{\widetilde}
\newcommand\lf{\bs}
\newcommand\bs{\backslash}

\newcommand\PSL{PSL_2(\Z)}

\newcommand\E{\mathcal E}
\newcommand\La{\Lambda}
\newcommand\GH{\Gamma}
\newcommand \PG{P\Gamma}
\newcommand\D{\mc D}
\newcommand\Hp{\mb H}
\newcommand\PAut{P\text{Aut}}
\newcommand\F{\mb F}
\newcommand\M{\mc M}
\newcommand\Aut{\text{Aut}}
\newcommand\Z{\mb Z}
\newcommand\C{\mb C}
\newcommand\tth{\theta}
\newcommand\w{\omega}
\newcommand\G{{G}}

\DeclareMathOperator{\disc}{disc}
\numberwithin{equation}{section}
\journalname{Geometriae Dedicata}
\begin{document}

\title{Periods of cubic surfaces with the automorphism group of order 54}

\titlerunning{Periods of cubic surfaces with the automorphism group of order 54}       

\author{Vasily Bolbachan}

\institute{Faculty of Mathematics, National Research University Higher School of Economics, 6 Usacheva Street, Moscow, Russia 119048\\
              \email{vbolbachan@gmail.com} 
              }

\date{}

\maketitle

\begin{abstract}
To any cubic surface, one can associate a cubic threefold given by a triple cover of $\mb P^3$ branched in this cubic surface. D. Allcock, J. Carlson, and D. Toledo used this construction to define the period map for cubic surfaces. It is interesting to calculate this map for some specific cubic surfaces. In this paper, we have calculated it in the case when the cubic surface is given by a triple covering of $\mb P^2$ branched in a smooth elliptic curve. In this case, the periods can be expressed through periods of the corresponding elliptic curve.
\end{abstract}

\section{Introduction}
\label{intro}
Let $S$ be a smooth cubic surface in $\mb P^3$. We can associate to it a cubic threefold $T=T(S)$ given by the triple cover of $\mb P^3$ branched at $S$. If the cubic surface $S$ is given as $\{(x,y,z,s)\in\mb P^3|g(x,y,z,s)=0\}$ then the equation of $T$ is $t^3+g(x,y,z,s)=0$. Let us denote by $\sigma$ the automorphism of $T$ given by the formula $(t,x,y,z,s)\mapsto (\w t,x,y,z,s)$, where $\omega=e^{2\pi i/3}$. 

The third complex cohomology of $T$ has a hermitian form $h'$ given by the following formula:
$$(\varphi,\psi)=\theta \int_{T}\varphi\wedge \overline \psi.$$

Here $\theta=\w-\w^2$. Since the third cohomology of $\mb P^3=T/\sigma$ is trivial, we can decompose $H^3(T,\mb C)$ with respect to the eigenvalues of $\sigma$: $H^3(T,\mb C)=H^3_\w(T,\mb C)\oplus H^3_{\w^2}(T,\mb C)$. It is easy to see that this decomposition is orthogonal with respect to $h'$. 

We recall that {\it the ring of Eisenstein numbers} $\E\subset \mathbb C$ is defined as a subring generated by $1$ and $\w$ over $\Z$. {\it A hermitian lattice} (over $\E$) is a finitely generated free $\E$-module together with an $\E$-valued hermitian form. In \cite[Section 2.3, Lemma 4.1]{allcock2002complex} was constructed a structure of hermitian lattice on $H^3(T,\Z)$ and was proved the following isomorphism of hermitian spaces: $(H^3(T,\Z)\otimes_\E\C,h)\simeq(H^3_{w^2}(T,\C),h')$. The hermitian lattice $H^3(T,\Z)$ is isomorphic to the standard unimodular lattice of signature $(4,1)$ called $\La$. It is a free $5$-dimensional module over $\E$ with the following hermitian form:
$$(x,y)_\La=-x_1\bar y_1+\sum\limits_{i=2}^5 x_i\bar y_i.$$

A subspace of a hermitian vector space is {\it negative} if the restriction of the hermitian form on it is negative definite. For a hermitian lattice of signature $(4,1)$ let us denote by $\D(R)$ the set of all $1$-dimensional negative subspaces and set $\mc D=\mc D(\La)$. One can prove that the subspace $H^{2,1}_{\w^2}(T,\C)\subset H^3_{\w^2}(T,\C)$ is negative and one-dimensional (see \cite[Lemma 2.6]{allcock2002complex}). In particular for any smooth cubic surface $S$, we have the canonical element $\lambda(S)$ in the space $\D(H^3(T,\Z))$.

{\it A framed cubic surface} is a pair $(S,r)$ consisting of a smooth cubic surface $S$ and an isomorphism of hermitian lattices $r\colon H^3(T,\Z)\to\Lambda$. By a convention, two isomorphisms differing by a unity of $\E$ give the same framing. Let us denote by $\wt\M$ the set of all framed cubic surfaces up to isomorphism. It has the natural structure of a smooth complex manifold. Let us denote by $\mc P'(S,r)\in \D$ the image of the element $\lambda(S)$ under the isomorphism $\mc D(H^3(T,\Z))\to\D$. It gives a well-defined map $\mc P'\colon\wt\M\to\D$  called {\it the Period map}.

For a vector $v\in \La$, denote by $\delta_v\subset \D$ the set of points orthogonal to $v$ under the hermitian form. Let $\mc H=\bigcup\limits_{N(v)=1}\delta_v$ and $\D'=\D\bs \mc H$. 
According to the first part of the Theorem 2.20 from \cite{allcock2002complex}, the image of the map $\mc P'$ is equal to $\D'$ and the map $\mc P'\colon \wt \M\to \D'$ is an isomorphism of complex manifolds.

Denote by $\Gamma$ the group consisting of automorphisms of the lattice $\La$ preserving the hermitian form and put $\PG=\Gamma/\mu_6$. The group $\PG$ acts on the space $\wt\M$ and the factor $\M:=\PG\bs\wt\M$ is isomorphic to the moduli space of smooth cubic surfaces. The map $\mc P'$ induces an isomorphism of complex analytic orbifolds $\mc P\colon \M\to \PG\bs\D'$. For a point $z\in \D'$ denote by $S_z$ the corresponding cubic surface.
\section{Statement of the main result}
The description of the groups arising as the automorphism group of a smooth cubic surface is well known (see \cite{dolgachev2012classical} and \cite{hosoh1997automorphism}). There are $12$ such groups, in total. In this list, there is a unique group of order $54$. Let us denote it by $\G$.
The group $\G$ can be realised as a subgroup of $GL_3(\E)$ generated by the permutation matrices and the diagonal matrix $(1,\w,\w^2)$. This group is known as the complex reflection group $G(3,3,3)$. 

Let us denote by $D_3$ the following hermitian lattice:
$$D_3=\{(x_1,x_2,x_3)\in\E^{\oplus 3}|x_1+x_2+x_3=0\mod \tth\}.$$
Let $R$ be a hermitian lattice. To any vector $v\in R$ of norm $2$ one can associate the reflection $\sigma_v$ in this vector. The group generated by these reflections is called {the Weyl group} of the lattice $R$ and is denoted by $W(R)$. In the next section we will prove that the group $W(D_3)$ is isomorphic to $\G$. In particular, if for some embedding $f\colon D_3\to \La$, the point $z\in\D$ is orthogonal to any vector of $f(D_3)$ then the stabilizer of the point $z$ in $\PG$ contains a group isomorphic to $\G$. This motivates the following statement:

\begin{proposition}
\label{prop:characterisation}
Let $z\in \D$. The following conditions are equivalent:
\begin{enumerate}
\item The point $z$ lies in $\D'$ and the stabilizer of this point in $\PG$ contains a subgroup isomorphic to $\G$.
    \item There is a primitive embedding $f\colon D_3\to \La$ such that for any $v\in f(D_3)$ we have $(z,v)_\La=0$.
\end{enumerate}
\end{proposition}

As we said above, the map $\mc P\colon \M\to\PG \lf\D'$ is an isomorphism of complex analytic orbifolds. In particular, the stabilizer of the point $z$ in $\PG$ is isomorphic to the automorphism group of the cubic surface $S_z$. This means that this proposition has the following corollary:

\begin{corollary}
\label{cor:characterisation}
The automorphism group of a smooth cubic surface $S$ has a subgroup isomorphic to $\G$ if and only if the point $\mc P(S)$ lies in the orthogonal complement to the image of the lattice $D_3$ under some primitive embedding $D_3\to\La$.
\end{corollary}

We will prove in the next section that up to the action of the group $\PG$ there is a unique primitive embedding $D_3\to\La$. For some specific choice of such an embedding its image is generated by the following vectors: $w_1=(-1,\w,-1,0,-1), w_2=(0,0,1,-1,0),w_3=(0,1,-1,0,0)$. We denote the lattice generated by these vectors by $L_0$. Let $\Hp$ be the upper half-plane. Since the lattice $D_3\simeq L_0$ is positive-definite, its orthogonal complement has signature $(1,1)$. So the orthogonal complement $(L_0)^{\perp}_{\mc D}$ to $L_0$ in $\D$ is isomorphic to $\Hp$. The explicit isomorphism between $\Hp$ and $(L_0)^{\perp}_{\mc D}$ is given by the following formula: $j_0(\tau)=\tau v_1+v_2$, where $v_1=(-\tth,\w^2,\w^2,\w^2,0), v_2=(1,0,0,0,1)$. 

Denote by $\Hp'$ the complement to the orbits of the points $i$ and $\omega$ in $\Hp$ and denote by $\D'_\G$ the set of points in $\D'$ whose stabilizer is isomorphic to $\G$. The following theorem gives the modular description of the moduli space of smooth cubic surfaces with the automorphism group of order $54$:
\begin{theorem}
\label{th:about_bijection}
The map $j_0$ induces a bijection
$$\bar j_0\colon PSL_2(\Z)\lf\Hp'\to \PG\lf\D'_\G.$$
\end{theorem}

Now we will apply this result to obtain the explicit formula for periods of some explicit family of cubic surfaces.

For $\tau\in \Hp$, denote by $E_\tau$ the elliptic curve given by the equation $P_\tau(x,y,z)=0$ in $\mathbb P^2$, where
$$P_\tau(x,y,z)=y^2z-4x^3+g_2(\tau)xz^2+g_3(\tau)z^3,$$
and $g_2(\tau)=60E_4(\tau), g_3(\tau)=140E_6(\tau)$ are multiples of the Eisenstein series. Denote by $X_\tau$ the cubic surface given by the equation $s^3+P_\tau(x,y,z)=0$ in $\mathbb P^3$. Since the space $\Hp$ is simply connected, a framing $q\colon H^3(T(X_i),\Z)\to \La$ of the cubic surface $X_i$ gives a canonical framing $r_q(\tau)$ of the cubic surface $X_{\tau}$ for any $\tau$. Define the map $\Pe_q'\colon \Hp\to \D'$ by the formula $\Pe_q'(\tau)=\mc P'((X_\tau, r_q(\tau)))$.

The following theorem gives the explicit formula for the periods of the cubic surface $X_\tau$. It provides a solution to the third problem posed in \cite[Section 6]{carlson2013cubic}. We recall that the group $\PG$ acts on $\D'$

\begin{theorem}
\label{th:explicit_formula}
For any framing $q$ of the cubic surface $X_i$, there is an element $g\in \PG$ such that for any $\tau\in \Hp$ we have $\mc P_q'(\tau)=g(j_0(\tau))$.  In particular $\mc P'(X_\tau)=j_0(\tau)\mod \PG$.
\end{theorem}

We remark that $X_i$ and $W_\w$ are the only cubic surfaces with the automorphism group of orders $108$ and $648$. The value of the period mapping $\mc P'$ for the cubic surface $X_\w$ was calculated in \cite[Section 11]{allcock2002complex}. So we can consider our result as a generalization of that calculation.

\section{Hermitian lattices} For background on hermitian lattices we refer reader to \cite{allcock1999new}. Denote  by $\overline{\cdot}$ the canonical involution on $\E$. Let $R$ be an arbitrary $\E$-module. \emph{An $\E$-valued hermitian form} (or briefly a hermitian form) on $R$ is a $\mathbb Z$-bilinear map $(\cdot,\cdot)\colon R\times R\to \E$ satisfying $(\alpha x,y)=\alpha (x,y)$ and $(x,y)=\overline{(y,x)}$ for any $\alpha\in \E, x,y\in R$. Denote by $K$ the field of fractions of $\E$. Similarly, one can define the notion of a $K$-valued hermitian form and a $K/\E$-valued hermitian form. A hermitian form is called \emph{nondegenerate} if for any non-zero $x\in R$ there is $y\in R$ such that $(x,y)\ne 0$.

\emph{A hermitian lattice} (or briefly \emph{a lattice}) is a pair consisting of a finitely generated free $\E$-module $R$ and a $\E$-valued hermitian form on $R$. Often we will denote this hermitian form by $(,)_R$ or just $(,)$. A hermitian lattice is nondegenerate if the corresponding hermitian form is nondegenerate. We will consider only nondegenerate lattices. \emph{A rank} of some lattice $R$ is by definition dimension of the vector space $R\otimes_\E K$ over $K$.

\emph{An extension of lattices $R_1\subset R_2$} is an hermitian lattice $R_2$ together with a sublattice $R_1$ such that $|R_2/R_1|<\infty$. The lattice $R_2$ is called \emph{overlattice} of the lattice $R_1$.

For a nondegenerate lattice $R$ {\it the dual lattice} $R^\vee$ is defined as the set of vectors $v\in R\otimes_\E K$ such that $(v,w)\in \E$ for any $w\in R$. The hermitian form on $R$ induces the canonical $K$-valued hermitian form on $R^\vee$.

\emph{A finite hermitian space} is a finitely generated torsion $\E$-module together with a $K/\E$-valued hermitian form. Let $W$ be some finite hermitian space. A subspace $W'\subset W$ is called isotropic if $(w,w')=0$ for any $w,w'\in W'$. A homomorphism of two finite hermitian spaces is a homomorphism of modules over $\E$ preserving the hermitian form. 

For an arbitrary hermitian lattice $R$ there is a canonical structure of a finite hermitian space on the $\E$-module $R^\vee/R$. This finite hermitian space is called \emph{the discriminant group of $R$}. We will denote it by $D(R)$. As in the case of $\Z$-valued lattices, there is a correspondence between the set of all overlattices $R'\supset R$ up-to isomorphism and the set of all isotropic subspaces $W\subset D(R)$. Moreover $|W|$ is equal to $|R'/R|$.

Choose some basis $e_i$ of $R$. The determinant of the matrix $ ((e_i,e_j)_R)_{i,j}$ does not depend on the choice of $e_i$. It is called \emph{the discriminant} of $R$. Denote it by $\disc R$. If we have an extension of lattices $R_1\subset R_2$ then $\disc R_2=|R_2/R_1|\disc R_1$. As in the case of $\mathbb Z$-valued lattices one can prove that $|D(R)|=(\disc R)^2$.

Denote by $\Aut(R)$ the set of all automorphisms of $R$ over $\E$ preserving the hermitian form. Let $\PAut(R)=\Aut(R)/\mu_6$ where $\mu_6$ is the group of the sixth roots of unity. In the same way, we can define the group $\Aut(W)$ for a finite hermitian space $W$. A lattice $R$ is called {\it unimodular} if $D(R)=0$. For a finite hermitian space $W$, denote by $W(-1)$ the same $\E$-module with the following hermitian form: $(v,w)_{W(-1)}=-(v,w)_W$. 
For a vector $v\in R$ of norm $2$, define the automorphism $\sigma_v$ by the formula:
$$\sigma_v(w)=w-(w,v)v.$$

Denote by $W(R)$ the group generated by $\sigma_v$ where $v$ ranges over all vectors of norm $2$. This group is called {\it the Weyl group} of $R$. It is easy to see that this group lies in the kernel of the natural map $\Aut(R)\to \Aut(D(R))$. In particular, for any lattice embedding $R'\to R$ there is a natural map $W(R')\to \Aut(R)$.

Denote by $H$ a $2$-dimensional hermitian lattice with the following hermitian form:
\begin{equation*}
  \left(\begin{matrix}
0 & \theta \\
-\theta & 0
\end{matrix}  
\right).
\end{equation*}
We set $V=D(H)$. It is easy to see that $V\simeq (\E/(\theta))^2$ with the following hermitian form
\begin{equation*}
  \left(\begin{matrix}
0 & \theta/3 \\
-\theta/3 & 0
\end{matrix}  
\right).
\end{equation*}

The following lemma summarizes the main properties of the lattice $H$:

\begin{lemma}
\label{lemma:about_H}
The following statements are true:
\begin{enumerate}
    \item We have $\Aut(V)\simeq SL_2(\F_3)$.
    \item We have $\Aut(H)\simeq SL_2(\Z)\times \mu_3$. The natural map $\Aut(H)\to \Aut(V)$ coincides with the natural projection $SL_2(\Z)\times \mu_3\to SL_2(\F_3)$.
    \item Let $H'$ be a hermitian lattice of signature $(1,1)$ such that $D(H')\simeq V$. Then $H'\simeq H$.
    \item The group $\Aut(H)$ acts transitively on the vectors of norm $-3$.
\end{enumerate}
\end{lemma}
\begin{proof}\begin{enumerate}
    \item We have $(av_1+bv_2,cv_1+dv_2)=(ad-bc)\tth/3$. So the matrix $\left(\begin{matrix}
a&b\\
c&d
\end{matrix}  
\right)$
lies in $\Aut(V)$ if and only if $ad-bc=1$.
\item The matrix $\left(\begin{matrix}
a&b\\
c&d
\end{matrix}  
\right)$ lies in $\Aut(H)$ if and only if $a\bar b=\bar ab,c\bar d=\bar c d,a\bar d-b\bar c=1$. The third equality implies that the numbers $a,b$ are coprime. So, from the first equality it follows that there is a unit $u_1$ such that $a=u_1\bar a, b=u_1\bar b$. In the same way, there is a unit $u_2$ such that $c=u_2\bar c, d=u_2\bar d$. Substituting these relations into the third equality, we get $u_1=u_2$. If the number $u_1$ belonged to the set $\{-1,-\w,-\w^2\}$ the numbers $a,b,c,d$ would be divisible by $\tth$ and so the third equality would not be satisfied. So $u_1\in \{1,\w,\w^2\}$. Since $\E\cap\mb R=\Z$ it follows that the matrix  $\left(\begin{matrix}
a&b\\
c&d
\end{matrix}  
\right)$ lies in $\mu_3\times SL_2(\Z)$. On the other hand, any such matrix lies in the group $\Aut(H)$.
\item Let $\wt H'\supset H'$ be the lattice corresponding to some isotropic vector of $V$. It is easy to see that $\wt H'$ is unimodular. According to \cite[the proof of Proposition 2.6.]{beauville2009moduli} any unimodular lattice of signature $(1,n), n\ge 1$ contains a non-zero isotropic vector. Applying this statement to the lattice $H'$, we conclude that $H'$ contains some non-zero isotropic vector. Call this vector $v_1$. We can assume that $v_1$ is primitive. There is a vector $v_2'\in H'^\vee$ such that $(v_1,v_2')=1$. Since $D(H')\simeq (\E/(\theta))^2$ the vector $v_2:=\theta v_2'$ lies in $H'$. Let us denote by $Q$ the lattice generated by $v_1$ and $v_2$. We have $\disc H'=|H'/Q|\disc Q$. Since $\disc Q=3$ and $(\disc H')^2=|D(H')|=|V|=9$ we conclude that $H'=Q$. Substituting $v_2\mapsto v_2+cv_1, c\in \E$, we can assume that $(v_2,v_2)\in \{-1,0,1\}$. It is easy to see that if the norm of the vector $v_2$ were non-zero, the finite hermitian space $D(H')$ wouldn't be  isomorphic to $V$. So $(v_2,v_2)=0$ and $H'\simeq H$.
\item Denote by $q_1,q_2$ the canonical basis of $H$. The norm of the vector $w=(a+b\omega)q_1+(c+d\omega)q_2$ is equal to $((a+b\omega)q_1+(c+d\omega)q_2, (a+b\omega)q_1+(c+d\omega)q_2)=(a+b\omega)\overline{(c+d\omega)}\theta-(c+d\omega)\overline{(a+b\omega)}\theta=((a+b\omega)(c+d\omega^2)-(c+d\omega)(a+b\omega^2))\theta=3(ad-bc)$. So this vector has norm  $-3$ if and only if $ad-bc=-1$. Now the statement follows from the following two facts:
\begin{enumerate}
    \item $SL_2(\mathbb Z)\subset \Aut(H)$.
    \item The group $SL_2(\mathbb Z)$ acts transitively on the pairs of vectors $((a,b),(c,d))\in \mathbb Z^2\times \mathbb Z^2$ satisfying $ad-bc=-1$.
\end{enumerate}
\end{enumerate}
\end{proof}

Let $a\in \E$. Denote by $D_3(a)$ the following hermitian lattice
$$\{(v_1,v_2,v_3)\in \E^{\oplus 3}|v_1+v_2+v_3=0\mod a\E\},$$
and put $D_3=D_3(\tth)$.

The following lemma summarizes the properties of the lattice $D_3$:
\begin{lemma}
\label{lemma:about_D3}
The following statements are true:
\begin{enumerate}
\item The finite hermitian space $D(D_3)$ is isomorphic to $V$.
    \item The Weyl group of the lattice $D_3$ is isomorphic to $\G$.
    \item The following sequence is exact:
    $$1\to W(D_3)\to \Aut(D_3)\to \Aut(D(D_3))\to 1.$$
\end{enumerate}
\end{lemma}

\begin{proof}\begin{enumerate}
    \item The group $D(D_3)$ is generated by the following elements: $$\alpha=(\tth/3,\tth/3,\tth/3), \beta=(1,0,0).$$ We have $(\alpha,\alpha),(\beta,\beta)\in \E, (\alpha,\beta)=\tth/3$. The first statement is proved.
    
    \item The reflections in the vectors $(1,-1,0),(0,1,-1)$ generate a subgroup isomorphic to $S_3$ acting as permutations of the coordinates. The reflection in the vector $(1,-\omega,0)$ is equal to the composition of some permutation and the diagonal matrix $(\w^2,\w,1)$. Denote by $G_0$ the group generated by these three reflections. By definition of the group $G$, the group $G_0$ is isomorphic to $G$. Let us prove that $G_0$ coincides with $W(D_3)$. The group $G_0$ acts on the set of vectors $v\in D_3$ of norm $2$. It is easy to see that this action is transitive. It follows that for any vector $v\in D_3$ of norm $2$, the reflection in this vector belongs to $G_0$. We conclude that $W(D_3)=G_0$ and so the group $W(D_3)$ is isomorphic to $G$.

    \item 
    \begin{enumerate}
        \item \emph{Exactness in the second term}. For any $g\in \Aut(D_3)$ lying in the kernel of the map $\Aut(D_3)\to \Aut(D(D_3))$ we need to show that $g\in W(D_3)$. Let us assume that we have an embedding of hermitian lattices $R_1\subset R_2$ and an automorphism of the lattice $R_1$. This automorphism can be extended to an automorphism of the lattice $R_2$ if and only if it preserves an isotropic subspace $W\subset D(R_1)$ corresponding to the embedding $R_1\subset R_2$. Applying this statement here, we conclude that $g$ has an extension to an automorphism of $\E^{\oplus 3}.$ It follows that $g$ has the form $g'C^kA^l, k\in\{0,1\}, l\in \{0,1,2\}$ where $g'\in W(D_3)$ and $C,A$ are the scalar matrices $(-1,-1,-1),(1,1,\w)$. To prove that $g\in W(D_3)$ we need to show that $C^kA^l=0$. Denote by $G'$ the subgroup generated by $C$ and $A$. It is enough to show that the map $G'\to \Aut(D(D_3))$ is injective. Let us identify the group $\Aut(D(D_3))$ with $SL_2(\mathbb F_3)$. Explicit computation shows that the images of the automorphisms $C$ and $A$ under the natural homomorphism $\Aut(D_3)\to \Aut(D(D_3))$ are equal to

$$\left(\begin{matrix}
1 &0\\
0 &-1
\end{matrix}\right), \quad
\left(\begin{matrix}
1 &0\\
1 & 1
\end{matrix}\right).$$

These two matrices generate a subgroup isomorphic to $\Z/6\Z$. We conclude that the natural map $G'\to \Aut(D(D_3))$ is injective and so $g\in W(D_3)$.

        \item \emph{Exactness in the third term}. For a vector $a$ of the lattice $\D_3$ of norm $3$ consider the following automorphism of the vector space $D_3\otimes_\E K$:

$$\sigma_{a}^3(v)=v-\dfrac {1-\w}3(v,a)a.$$

If the vector $a$ satisfies $(a,w)\in \theta \E$ for any $w\in D_3$, then the automorphism $\sigma_{a}^3(v)$ lies in the group $\Aut(D_3)$. Let $a_1=(0,0,\tth), a_2=(1,1,\w)$. The automorphism $\sigma_{a_1}$ is equal to $A$ and the automorphism $\sigma_{a_2}^3$ has the following form: 

$$\sigma_{a_2}=\dfrac{-\tth}3\left(\begin{matrix}
  \w&\w^2&1\\
  \w^2&\w&1\\
  \w&\w&\w
\end{matrix}\right).
$$

%

 Explicit calculation shows that the automorphisms $\sigma_{a_1}^3, \left({\sigma_{a_1}^3}\right)^2\sigma_{a_2}^3$ have the following images in the group $\Aut(D(D_3))$:
 $$\left(\begin{matrix}
1 &0\\
1 &1
\end{matrix}\right), 
\left(\begin{matrix}
0 & -1\\
1 & 0
\end{matrix}\right).$$

It is well known that these matrices generate the whole group $SL_2(\F_3)$. So the map $\Aut(D_3)\to \Aut(D(D_3))$ is surjective.
    \end{enumerate}

\end{enumerate}
\end{proof}

\begin{remark}
It is worth mentioning that the automorphisms $\sigma_{a_1}^3, \sigma_{a_2}^3$ generate the stabilizer of the vector $(-1,1,0)$ and this stabilizer is isomorphic to the group $SL_2(\F_3)$. So the group $\Aut(D_3)$ is isomorphic to the semi-direct product $\G\rtimes SL_2(\F_3)$. It follows from the proof that the image of the group $\Aut(D_3)$ under the natural map $GL_3(\mb C)\to PGL_3(\mb C)$ is equal to the so-called Hessian group (see \cite{dolgachevHesse}). It follows from the proof that the group $\Aut(D_3)$ is a complex reflection group of order $1296$. It has the number $26$ in the Shephard-Todd's list \cite{ShephardTodd}.
\end{remark}

\begin{proposition}
\label{prop:primitive_embedding}
There is a primitive embedding $D_3\to\La$. The action of the group $\Aut(\La)$ on the set of all such embeddings is transitive.
\end{proposition}

\begin{proof}
Let us denote by $\La'$ the direct sum $D_3\oplus H$. We have $D(\La')\simeq V^{\oplus 2}$. Since $V\simeq V(-1)$, $D(\La')\simeq V\oplus V(-1)$. Let us denote by $\La''$ the extension of the lattice $\La'$ corresponding to the isotropic subspace $V'=\{(v,v)|v\in V\}\subset V\oplus V(-1)$. (Although V is isomorphic to $V'$ as $\E$-module, the restriction of the hermitian form on $V'$ is zero). We have $|\La''/\La'|=|V'|=9$. So $\disc \La=(\disc\La')/|\La''/\La'|=1$ and we conclude that the lattice $\La''$ is unimodular. According to \cite[Theorem 7.1.]{allcock1999new}, any two unimodular lattices of signature $(1,4)$ are isomorphic. In particular $\La''\simeq \La$ and we have constructed an embedding $D_3\subset D_3\oplus H\to\La$. It is easy to see that this embedding is primitive.

Let $L$ be the image of an arbitrary primitive embedding $D_3\to\La$. According to Lemma \ref{lemma:about_H} and Lemma \ref{lemma:about_D3} the orthogonal complement of $L$ inside $\La$ is isomorphic to $H$. Since there is a bijective correspondence between overlattices of $D_3\oplus H$ and isotropic subspaces $W\subset D(D_3\oplus H)$, it is enough to show that the group $\Aut(D_3\oplus H)$ acts transitively on the set of the subspaces $W\subset V\oplus V(-1)$ arising from some primitive embedding $D_3\oplus H\to \Lambda$. By Lemma \ref{lemma:about_H} and Lemma \ref{lemma:about_D3} the map $\Aut(D_3\oplus M)\to \Aut(D(D_3\oplus M))$ is surjective. So it is enough to show that the group $\Aut(V\oplus V(-1))$ acts transitively on the set of all isotropic subspaces arising from some primitive embedding $D_3\oplus H\to \Lambda$.

Since the embedding is primitive, the isotropic subspace $W$ does not contain any vectors of the form $(v,0), (0,v), v\ne 0$. We have $$|W|=|\Lambda/(L\oplus L_\La^\perp)|=\disc(L\oplus L_\La^\perp)/\disc(\La)=\disc{D_3}\disc(H)=9.$$
So $W$ is a $2$-dimensional vector space over $\mathbb F_3$.
Let us identify $V\oplus V(-1)$ with $\mathbb F_3^{\oplus 4}$. The hermitian form has the form $((a_1,a_2,a_3,a_4),(b_1,b_2,b_3,b_4))=((a_1b_2-a_2b_1)-(a_3b_4-a_4b_3))\theta/3$. Since $W$ does not contain any vectors of the form $(v,0), (0,v), v\ne 0$, it has a basis having the following form: $e_1=(1,0,a,b), e_2=(0,1,c,d)$. Since this space is isotropic, we have $ad-bc=1$. So there is an element $g$ of the group $SL_2(\mathbb F_3)$ such that $g((1,0))=(a,b), g((0,1))=(c,d)$. We conclude that $W$ is equal to $(1\times g)W_0$, where $W_0=\{(a,b,a,b)|a,b\in \mathbb F_3\}$ and $1\times g\in SL_2(\mathbb F_3)\times SL_2(\mathbb F_3)$.

Since $\Aut(V\oplus V(-1))\cong SL_2(\mathbb F_3)\times SL_2(\mathbb F_3)$, we conclude that the group $\Aut(V\oplus V(-1))$ acts transitively on the set of all isotropic subspaces arising from some primitive embedding $D_3\oplus H\to \Lambda$.  The proposition is proved.
\end{proof}

\begin{proof}[The proof of Proposition \ref{prop:characterisation}]
Let us prove the following implications
\begin{itemize}
    \item[$1\Rightarrow 2$.] As it was stated above, the stabilizer of the point $z$ in the group $\PG$ is isomorphic to the automorphism group of the cubic surface $S_z$. We can consider the automorphism group of a cubic surface $S_z$ as a subgroup of the Weyl group of the root system $E_6$ called $W(E_6)$(see for example \cite[Section 9.5]{dolgachev2012classical}). A conjugacy class $C$ of $W(E_6)$ is called  {\it effective} if there is a smooth cubic surface $S$ and an automorphism $g$ acting on it, such that $g$ belongs to $C$.
    
    According to \cite[Proposition 5.2]{hosoh1997automorphism} there are $16$ effective conjugacy classes in total. An automorphism of $S\subset \mb P^3$ is called {\it a harmonic homology} if it can be extended to an automorphism of $\mb P^3$ of order $2$ which fixes some hyperplane. According to \cite[Section 9.5.1]{dolgachev2012classical} an automorphism $g$ is a harmonic homology if and only if it lies in the conjugacy class $(4A_1)$ in the terminology of Cartier.

    According to \cite{hosoh1997automorphism} the group $\G$ is generated by three elements lying in the conjugacy class $(4A_1)$. By Lemma 11.4 of \cite{allcock2002complex} they correspond to reflections in some vectors of norm $2$  in $\La$. Let us denote them as $v_1,v_2,v_3$. So the reflections in these vectors generate a subgroup of the stabilizer of the point $z$ isomorphic to $\G$.

    Denote by $L$ the lattice generated by these vectors. It is enough to prove that $L\simeq D_3$. Let  $g_i$ be the reflection in the vector $v_i$. Since $(g_1g_2)^3=(g_1g_3)^3=(g_2g_3)^3=1$, all scalar products $(v_i,v_j),1\leq 1<j \leq 3$ are units in $\E$. Multiplying them by the sixth roots of unity and rearranging them, we can assume that  $(v_1,v_2)=1,(v_2,v_3)=1,(v_1,v_3)=a, (v_i,v_i)=2, 1\leq i\leq 3, a\in \{1,-1,w,-w\}$. There are four cases:
    \begin{enumerate}
        \item Let $a=-\w$. In this case $L\simeq D_3$ and the statement is proved.
        \item Let $a=\omega$. In this case the vector $v_3-v_2+v_1$ has norm $1$. This means that the point $z$ lies in the orthogonal complement to some vector of norm $1$ and so $z\notin\D'$.
        \item Let $a=-1$. In this case the vector $v_3-v_2+v_1$ has norm $0$. So this vector is equal to zero and the lattice $L$ is generated by the vectors $v_1$ and $v_2$. In this case it is easy to see that the group $W(L)$ is isomorphic to $S_3$.
        \item Let $a=1$. In this case the lattice $L$ isomorphic to the lattice $D_3(2)$. It is easy to see that in this case we have $W(L)\simeq S_3\rtimes (\Z/2\Z)^3$ and so $|W(L)|<|\G|$.

    \end{enumerate}
    \item[$1\Leftarrow 2$.] Since the group $W(D_3)$ is generated by reflections, the stabilizer of the point $z$ contains a subgroup isomorphic to $W(D_3)\simeq \G$. So it is enough to prove that $z\in \D'$. Let us assume that this is not true. Then there is some vector $v_4$ of norm $1$ such that the point $z$ lies in the orthogonal complement to the lattice generated by $L$ and $v_4$. Let us denote this lattice by $L'$. Since the lattice $L$ is primitively embedded in $\La$ and $v_4\notin L$, rank of the lattice $L'$ is equal to $4$. Choose the following basis in the lattice $L$: $v_1=(1,-1,0), v_2=(0,1,-1), v_3=(0,0,\tth)$. Since this lattice is positive definite, the number $\det ((v_i,v_j)),1\leq i,j\leq 4$ is positive. We set $(v_i,v_4)=a_i, 1\leq i\leq 4$. The direct computation shows that this determinant is equal to
    $$3-Q((a_1,a_2,a_3),(a_1,a_2,a_3))$$
    where $Q$ is the hermitian form with the following matrix:
    $$Q=\left(\begin{matrix}
6 &3 &\tth\\
3 & 6 &2\tth\\
-\tth & -2\tth & 3
\end{matrix}\right).$$
It is easy to see that the form $Q$ is positive definite. Since all the elements of this matrix lie in the ideal $\tth\E$, for any non-zero $(a_1,a_2,a_3)$ we have $Q((a_i),(a_i))\ge 3$. So the determinant is positive if and only if $a_1=a_2=a_3=0$. In this case $v_4\in L^{\perp}\simeq H$. But the lattice $H$ does not contain any vectors of norm $1$.
\end{itemize}
\end{proof}

We recall that the group $\PG$ is defined as $\PAut(\Lambda):=\Aut(\Lambda)/\mu_6$. Denote by $L$ the image of the lattice $D_3$ under some primitive embedding $D_3\to \Lambda$. The following proposition determines the structure of the normalizer of the group $W(L)$ inside the group $\PG$.
\begin{proposition}
\label{prop:normalizer}
Let us denote the image of the lattice $D_3$ under some primitive embedding $D_3\to\La$ by $L$ and by $M$ its orthogonal complement. The normalizer $N_{\GH}(W(L))$ is equal to the subgroup of the group $\Aut(L)\times \Aut(M)$ consisting of  elements having an extension to $\GH$. The natural projection $\Aut(L)\times \Aut(M)\to \Aut(M)$ gives the following exact sequence:
$$1\to\G\to N_{\GH}(W(L))\to\mu_3\times SL_2(\Z)\to 1.$$
In particular $N_{\PG}(W(L))/W(L)\simeq PSL_2(\Z)$.
\end{proposition}
\begin{proof}
If an element of the group $\GH$ lies in the normalizer $N_{\GH}(W(L))$, then it preserves the set of vectors of the lattice $L$ of norm $2$. Hence it preserves the lattice $L$ and its orthogonal complement $M$. The first statement is proved. 

From the proof of Proposition \ref{prop:primitive_embedding} it follows that $D(L\oplus M)\simeq V(-1)\oplus V$ and the subspace corresponding to the extension $L\oplus M\subset \Lambda$ is equal to $V'=\{(v,v)|v\in V\}\subset V(-1)\oplus V$. Consider the following two maps:
$$a\colon \Aut(L)\to \Aut(D(L)), b\colon \Aut(M)\to \Aut(D(M)).$$

We have $\Aut(D(L))\simeq \Aut(D(M))\simeq SL_2(\F_3)$. It follows that the element $(g,h)\in \Aut(L)\times \Aut(M)$ has some extension to the group $\GH$ if and only if $a(g)=b(h)$. The kernel of the map $N_{\GH}(W(L))\to \Aut(M)$ coincides with the kernel of the map $a$. According to Lemma \ref{lemma:about_D3} this kernel is equal to $W(L)$. The image of the map $N_{\GH}(W(L))\to \Aut(M)$ is equal to the whole $\Aut(M)$ because the maps $a,b$ are surjective by Lemma \ref{lemma:about_H} and Lemma \ref{lemma:about_D3}.

The last statement follows from the fact that the restriction of the map $$N_{\PG}(W(L))\to \mu_3\times SL_2(\Z)$$ to the sixth roots of unity is injective.
\end{proof}

\section{Modular variety} Denote by $L_0\subset \La$ a lattice generated by the following vectors: $$w_1=(-1,\w,-1,0,-1), w_2=(0,0,1,-1,0),w_3=(0,1,-1,0,0).$$ and by $M_0$ its orthogonal complement in $\La$. It easy to see that the lattice $M_0$ is generated by the following vectors: $v_1=(-\tth,\w^2,\w^2,\w^2,0), v_2=(1,0,0,0,1)$. We recall that we have defined an isomorphism $j_0\colon \Hp\to (L_0)^{\perp}_{\mc D}=\D(M_0)$ given by the formula $j_0(\tau)=\tau v_1+v_2$. 

It is easy to see that the lattice $L_0$ is isomorphic to $D_3$ and primitively embedded in $\La$.  Denote by $X_{\ge W(L_0)}$ the set of points $z\in \D$ which are invariant under the group $W(L_0)$. Since any element in $X_{\ge W(L_0)}$ is orthogonal to any vector of $L_0$, it is equal to the intersection $M_0\otimes_\E\mb C\cap \D$. Denote by $X_{W(L_0)}$ the set of points $z\in \D$ whose stabilizer is equal to $W(L_0)$.  We have the inclusion $X_{W(L_0)}\subset X_{\ge W(L_0)}$ and we will prove that $X_{W(L_0)}$ is an open subset of $X_{\ge W(L_0)}$.

The map $j_0$ induces a biholomorphism $\Hp\to X_{\ge W(L_0)}$. Denote by $j_0^{-1}$ the inverse map $X_{\ge W(L_0)}\to \Hp$.

The group $\mu_3\times SL_2(\Z)$ acts on $\Hp$ via the natural homomorphism $$\mu_3\times SL_2(\Z)\to PSL_2(\Z)$$ and it acts on the space $X_{\ge W(L_0)}$ via the identification $\mu_3\times SL_2(\Z)=\Aut(M_0)$. It is easy to see that the map $j_0$ is equivariant under these actions.

Denote by $\Hp'$ the complement to orbits of the points $i$ and $\w$ under the group $PSL_2(\Z)$ and denote by $Stab_{\PG}(z)$ the stabilizer group of the point $z\in \D$ in the group $\PG$.

\begin{proposition}
\label{prop:half_plane}
For a point $z\in X_{\ge W(L_0)}$ we put $z'=j_0^{-1}(z)$. There are the following three possibilities:
\begin{enumerate}
\item The point $z'$ is conjugate to the point $\w$. In this case the group $Stab_{\PG}(z)$ has order $648$.
    \item The point $z'$ is conjugate to the point $i$. In this case the group $Stab_{\PG}(z)$ has order $108$.
    \item In the remaining case, the group $Stab_{\PG}(z)$ is isomorphic to $\G$. 
\end{enumerate}
In particular, the map $j_0$ induces a biholomorphism $\Hp'\to X_{W(L_0)}$.
\end{proposition}

\begin{proof}
According to the classification from \cite{hosoh1997automorphism} there are three cases:
\begin{enumerate}
\item $|G_z|=648$. According to the classification from \cite{hosoh1997automorphism} and the proof of Theorem 11.9 from \cite{allcock2002complex}, this is equivalent to the statement that the vector $z$ is proportional to some vector $\wt z$ lying in $\La$ and having norm $-3$. Since the lattice $M_0$ is primitively embedded in $\La$ it follows that $\wt z\in M_0$. It is easy to see that the vector $\omega v_1+v_2$ has norm $-3$. So it is enough to prove that the group $\Aut(M_0)$ acts transitively on the vectors of norm $-3$. This statement follows from the fourth statement of Lemma \ref{lemma:about_H} (we recall that $M_0\cong H$). 

    \item $|G_z|=108$.  In this case the elements of the group $G_{z}$ lying in the conjugacy class $(4A_1)$ generate a subgroup isomorphic to $\G$. Since this subgroup is normal, the group $G_z$ is contained in the group $N_{\PG}(W(L))$. So the stabilizer of the point $z$ in $N_{\PG}(W(L))$ has $108$ elements. According to \cite[Section 6.4. (10)]{hosoh1997automorphism} the group $Stab_{\PG}(z)$ is generated by $\G$ and a certain element $h$ satisfying $h\not\in G, h^4=1, h^2\in\G$. This means that the point $z'$ is invariant under some element of the group $PSL_2(\Z)$ of order $2$. So the point $z'$ is conjugate to $i$.
    
    Let us assume that the point $z'$ is conjugate to $i$. Then the point $z'$ is fixed under some non-trivial element from the group $PSL_2(\Z)$. So the stabilizer of the point $z$ in $\PG$ is strictly greater than $\G$. The point $z$ is not proportional to any element from $\Lambda$. By (1) of this proof, it follows that  $|Stab_{\PG}(z)|\ne 648$. So because of the classification from \cite{hosoh1997automorphism} the order of the group $Stab_{\PG}(z)$ is equal to $108$.
    
    \item $|G_z|=54$. According to the two previous cases, the point $z'$ lies in $\Hp'$.
\end{enumerate}
\end{proof}
We recall that the cubic surface $X_{\tau}$ was defined in the second section. We have the following statement:

\begin{lemma}
\label{lemma:Hasse}
There are the following three possibilities:
\begin{enumerate}
    \item If the point $\tau$ is conjugate to the point $\w$ then $|Aut(X_\tau)|=648$.
    \item If the point $\tau$ is conjugate to the point $i$ then $|Aut(X_\tau)|=108$.
    \item In the remaining case $Aut(X_\tau)\simeq \G$.
\end{enumerate}
\end{lemma}

\begin{proof}
Any cubic curve is projectively equivalent to a cubic curve in the so-called Hesse form \cite[Lemma 1]{dolgachevHesse}. This means that it is equivalent to the curve $E'_{\lambda}$, where the curve $E'_{\lambda}$ is given by the equation $x^3+y^3+z^3+6\lambda xyz=0$ in $\mb P^2$. According to the proof of Lemma $1$ from \cite{dolgachevHesse}, the $j$ invariant of the curve $E'_{\lambda}$ is equal to $1728\cdot 4A^3/(4A^3+27B^2)$ where $A=12\lambda(1-\lambda^3), B=2(1-20\lambda^3-8\lambda^6))$. Let us denote by $X'_{\lambda}$ a cubic surface given by the equation $s^3+x^3+y^3+z^3+6\lambda xyz=0$ in $\mathbb P^3$. According to \cite[Theorem 9.5.8]{dolgachev2012classical} the smooth cubic surface $X'_{\lambda}$ has the automorphism group of order $108$ if and only if $1-20\lambda^3-8\lambda^6=0$, has the automorphism group of order $648$ if and only if $\lambda=\lambda^4$ and in the remaining case, it has the automorphism group isomorphic to $G$.  Let us consider the following three cases:

\begin{enumerate}
    
 \item The point $\tau$ is conjugate to $\omega$. In this case $j$-invariant of the cubic curve $E_\tau$ is equal to $0$. So if we write this curve in Hesse form we would have $A=0$ and hence $\lambda=\lambda^4$.  The curves $E_\tau, E_\lambda'$ are projectively equivalent. It follows that $X_\tau\cong X_\lambda'$. So it follows from \cite{dolgachev2012classical} that  $|\Aut(X_\tau)|=648$.
 
  \item The point $\tau$ is conjugate to $i$. In this case $j$-invariant of the cubic curve $E_\tau$ is equal to $1728$. So if we write this curve in Hesse form we would have $B=0$ and hence $1-20\lambda^3-8\lambda^6=0$.  As in the previous case we conclude that $|\Aut(X_\tau)|=108$.

        \item The remaining case. In this case $j$-invariant of the cubic curve $E_\tau$ is not equal to $0$ or $1728$. So if we write this curve in Hesse form we would have $A, B\ne 0$. So it follows from \cite{dolgachev2012classical} that $\Aut(X_\tau)\cong G$.
\end{enumerate}
\end{proof}

\begin{lemma}
\label{lemma:connectedness}The following statements hold
\begin{enumerate}
 \item For any $x,y\in X_{W(L_0)}$ if some $\gamma\in \PG$ satisfies $\gamma x=y$ then $\gamma\in N_{\PG}(W(L_0))$.
 \item If $g_1, g_2$ are two different elements in $\PG/N_{\PG}(W(L_0))$ then the two sets $g_1X_{W(L_0)}, g_2X_{W(L_0)}$ do not intersect.
    \item Let us consider a subset $\D'_\G=\bigcup\limits_{g\in \PG/N_{\PG}(W(L_0))}gX_{W(L_0)}$ of $\D$ with the induced topology. For each $g_0\in\PG/N_{\PG}(W(L_0))$ the subset $g_0X_{W(L_0)}$ is open and closed in $\D'_\G$.
\end{enumerate}
\end{lemma}

\begin{proof}
\begin{enumerate}
\item Let $g\in W(L_0)$. Since $gy=y$ we have
    $$\gamma^{-1} g \gamma x=\gamma^{-1} g y=\gamma ^{-1}y=x.$$
    Since the stabilizer of the point $x$ is equal to $W(L_0)$ it follows that $\gamma^{-1}g\gamma\in W(L_0)$. Since it is true for any $g\in W(L_0)$ we have $\gamma\in N_{\PG}(W(L_0))$.
    \item If they did intersect there would be points $x,y\in X_{W(L_0)}$ and elements $g_1,g_2\in \PG$ such that $g_1x=g_2y$. So $x=g_1^{-1}g_2y$. According to the first statement of this lemma the element $g_1^{-1}g_2$ lies in the subgroup $N_{\PG}(W(L_0))$. This means that $g_2=g_1w,w\in N_{\PG}(W(L_0))$ and so $g_1N_{\PG}(W(L_0))=g_2N_{\PG}(W(L_0))$. A contradiction.
\item Since the set $g_0X_{\ge W(L_0)}$ is closed in $\D$, the closure of the set $g_0X_{W(L_0)}$ in $\D'_\G$ is contained in the set $g_0X_{\ge W(L_0)}\cap \D'_\G=g_0X_{W(L_0)}$. So the set $g_0X_{W(L_0)}$  is closed in $\D'_\G$.

For a vector $v\in\Lambda$ of norm $2$ denote by $\delta_v$ the set of points $z\in\D$ which are orthogonal to $v$. Similarly to the proof of Lemma 5.3 from \cite{beauville2009moduli} one can show that the family of subsets $\{\delta_v\}_{(v,v)=2}$ is locally finite. Since the set $gX_{\ge W(L_0)}$ is an intersection of a finite number of sets of the form $\delta_v$, the family of subsets $\{gX_{\ge W(L_0)}\}_{ g_0\ne g\in \PG/W(L)}$ is locally finite too. So the set $$Y=\bigcup\limits_{g_0\ne g\in \PG/N_{\PG}(W(L_0))}gX_{W(L_0)}$$ is closed in $\D'_\G$. Therefore the complement $\D'_\G\bs Y$ is open. By the previous statement of the lemma this complement coincides with $g_0X_{W(L_0)}$.
\end{enumerate}
\end{proof}

\begin{proof}[The proof of Theorem \ref{th:about_bijection}]
Since the map $j_0\colon \Hp'\to \D'_\G$ can be decomposed as $\Hp'\to X_{W(L_0)}\to \D'_\G$ it is enough to prove that the following two maps are well defined and bijective:
\begin{enumerate}
    \item The map $f_1\colon \PSL\lf \Hp'\to N_{\PG}(W(L_0))\lf X_{W(L_0)}$. If two points $\tau_1,\tau_2$ are equivalent under $PSL_2(\Z)$, the points $j_0(\tau_1), j_0(\tau_2)$ are equivalent under the group $\PAut(M_0)\simeq PSL_2(\Z)$. By Proposition \ref{prop:normalizer} the map $N_{\PG}(W(L_0))\to \Aut(M_0)$ is surjective. So the points $j_0(\tau_1), j_0(\tau_2)$ are equivalent under the group $N_{\PG}(W(L_0))$. We have proved that the map $f_1$ is well-defined and injective. By Proposition \ref{prop:half_plane} the map $\Hp'\to X_{W(L_0)}$ is bijective. So the map $f_1$ is surjective.
    \item The map $f_2\colon N_{\PG}(W(L_0))\lf X_{W(L_0)}\to \PG\lf\D'_\G$. Obviously, this map is well defined. It is injective by the first statement of Lemma \ref{lemma:connectedness}. By Proposition \ref{prop:characterisation}, Proposition \ref{prop:primitive_embedding} and Proposition \ref{prop:half_plane} it is surjective.
\end{enumerate}
\end{proof}

The following lemma is well-known.

\begin{lemma}
\label{lemma:extension}
Let $f$ be an injective holomorphic map from $\mathbb P^1\bs\{0,1,\infty\}$ to itself. Then $f$ is a linear fractional transformation.
\end{lemma}
\begin{proof}
Let us denote by $Y$ the image of $f$. By the open mapping theorem, $f$ induces a homeomorphism from $\mathbb P^1\bs\{0,1,\infty\}$ to $Y$.  Let us prove that this function has poles at the points $0,1,\infty$. By contradiction, let us assume that the function $f$ had an essential singularity at, say, the point $0$. Let us choose some open bounded neighborhood $U$ of the point $0$. Let $U'=U\bs (U\cap\{0,1\})$. By the Casorati-Weierstrass theorem the image $W:=f(U')$ is dense in $\mathbb P^1\bs\{0,1,\infty\}$. So the set $W$ is dense in $Y$ as well. Since the map $f\colon \mathbb P^1\bs\{0,1,\infty\} \to Y$ is a homeomorphism, we conclude that $U'$ is dense in $\mathbb P^1\bs\{0,1,\infty\}$. A contradiction. This means that $f$ is a rational function. Since the function $f$ is injective, $f$ is a linear fractional transformation.
\end{proof}

\begin{proof}[Proof of Theorem \ref{th:explicit_formula}]The map $\mc P_q'$ satisfies the following equation: $\mc P_{gq}'(\tau)=g\mc P_{q}'(\tau)$. In particular, it is enough to prove the statement of the theorem for some specific $q$. So we can assume that for some $\tau\in\Hp'$ the point $\mc P_q'(\tau)$ lies in $X_{W(L_0)}$. According to Lemma \ref{lemma:Hasse}, the map $\mc P_q'$ induces a map $\mc P_q'|_{\Hp'}\colon \Hp'\to \D'_\G$. 
Since the set $\Hp\bs \Hp'$ is discrete in $\Hp$ the set $\Hp'$ is path-connected. So, according to the third statement of Lemma \ref{lemma:connectedness} the image of $\Hp'$ under the map $\mc P_q'|_{\Hp'}$ is contained in the set $X_{W(L_0)}$.  Since $X_{\ge W(L_0)}$ is closed and $\Hp'$ is dense in $\Hp$ the image of $\mc P_q'$ is contained in $X_{\ge W(L_0)}$. 

Denote $f:=j_0^{-1} \circ\mc P_q'$. Let us prove that the map $f$ has the form $f(\tau)=\dfrac {a\tau+b}{c\tau+d}$ for some $a,b,c,d\in \Z, ad-bc=1$.

 If two points $\tau_1,\tau_2$ are equivalent under the group $\Hp'$, the cubic surfaces $X_{\tau_1}, X_{\tau_2}$ are isomorphic. It implies that the points $\mc P_q'(\tau_1), \mc P_q'(\tau_1)$ are equivalent under the group $N_{\PG}(W(L_0))$ and therefore the points $f(\tau_1), f(\tau_2)$ are equivalent under the group $PSL_2(\Z)$. So the map $f$ induces a well-defined injective map $$\bar f\colon \PSL\lf\Hp'\to \PSL\lf\Hp'.$$ Let us identify the space $\PSL\lf\Hp'$ with $\mb P^1\bs\{0,1,\infty\}$ using the $j$-invariant. 
 By Lemma \ref{lemma:extension}, the map $\bar f$ is a linear fractional transformation. Denote by $\wt f$ its extension to $\mathbb P^1$. We have $\wt f(\mb P^1\bs\{0,1,\infty\})\subset \mb P^1\bs\{0,1,\infty\}$. By Proposition \ref{prop:half_plane} and Lemma \ref{lemma:Hasse} the map $f$ preserves the orbits of the points $i$ and $\w$. So $\wt f(0)=0, \wt f(1)=1$. It follows that the map $\wt f$ is identical.

The map $j|_{\Hp'}\colon \Hp'\to \mb P^1\bs\{0,1,\infty\}$ is a topological covering. It follows that the map $f|_{\Hp'}$ lies in $PSL_2(\Z)$. Therefore the map $f$ also lies in $PSL_2(\Z)$. It follows that the map $\mc P_q'$ is given by the formula $\mc P_q'(\tau)=g_1(j_0(g_2(\tau)))$, for some $g_1\in \PG, g_2\in PSL_2(\Z)$. Since any automorphism of the lattice $M_0$ can be extended to an automorphism of $\La$, the theorem is proved.
\end{proof}

\begin{acknowledgements}
I thank Valery
Gritsenko for his interest in this paper. I also very grateful to the referee who made a lot of useful remarks.
\end{acknowledgements}

\bibliographystyle{spmpsci.bst}      
\bibliography{mylib}   

%
%

\end{document}